\documentclass[10pt,twocolumn,amsmath,amssymb,aps,pra,secnumarabic,
    nofootinbib,groupedaddress,floatfix]{revtex4-1}

\usepackage{mathtools,mathptmx,amsthm,tikz,enumitem,eucal}
\usepackage[colorlinks=true,linkcolor=red!70!black,citecolor=blue!70!black,
    urlcolor=magenta!70!black,backref=page]{hyperref}
\setlist[enumerate,1]{font=\bfseries,label=\arabic*.}
\setlist[enumerate,2]{font=\bfseries,label=(\alph*)}
\numberwithin{equation}{section}

\usepackage{dingbat}

\makeatletter\def\@bibdataout@init{}\def\pre@bibdata{}\makeatother

\counterwithin{figure}{section}

\usetikzlibrary{arrows,shadings,calc}
\usetikzlibrary{decorations.pathreplacing}

\tikzstyle{basic}=[thick,scale=.35,baseline=-.7ex]

\tikzstyle{overcross}=[draw=white,double=black,line width=2pt,
    double distance=0.8pt] 

\newenvironment{tp}{\begin{tikzpicture}}{\end{tikzpicture}}

\colorlet{darkred}{red!70!black}
\colorlet{medred}{red!85!black}
\colorlet{medblue}{blue!85!black}
\colorlet{medyel}{yellow!80!black}
\colorlet{medmag}{magenta!70!black}
\colorlet{medcyan}{cyan!80!black}
\colorlet{medgreen}{green!75!black}

\newtheorem{theorem}{Theorem}[section]
\newtheorem{corollary}[theorem]{Corollary}
\newtheorem{lemma}[theorem]{Lemma}
\theoremstyle{definition}

\theoremstyle{remark}
\newtheorem*{remark}{Remark}
\newtheorem*{example}{Example}

\newcommand{\Cor}[1]{Corollary~\ref{#1}}

\newcommand{\Fig}[1]{Figure~\ref{#1}}
\newcommand{\Lem}[1]{Lemma~\ref{#1}}
\newcommand{\Sec}[1]{Section~\ref{#1}}
\newcommand{\Thm}[1]{Theorem~\ref{#1}}
\newcommand{\equ}[1]{equation~\eqref{#1}}

\newcommand{\ie}{\emph{i.e.}}

\newcommand{\Z}{\mathbb{Z}}
\newcommand{\Q}{\mathbb{Q}}

\newcommand{\cI}{\mathcal{I}}
\newcommand{\cJ}{\mathcal{J}}
\newcommand{\cK}{\mathcal{K}}
\newcommand{\cL}{\mathcal{L}}
\newcommand{\cP}{\mathcal{P}}
\newcommand{\cO}{\mathcal{O}}
\newcommand{\cT}{\mathcal{T}}

\newcommand{\bcO}{\overline{\cO}}

\newcommand{\AS}{\text{AS}}
\newcommand{\Br}{\text{Br}}
\newcommand{\CL}{\text{CL}}
\newcommand{\Di}{\text{Di}}
\newcommand{\IHX}{\text{IHX}}

\newcommand{\Aut}{\operatorname{Aut}}
\newcommand{\hAut}{\operatorname{hAut}}
\newcommand{\Out}{\operatorname{Out}}

\newcommand{\Coinv}{\operatorname{Coinv}}
\newcommand{\Inv}{\operatorname{Inv}}

\newcommand{\GL}{\operatorname{GL}}
\newcommand{\SL}{\operatorname{SL}}

\newcommand{\onto}{\twoheadrightarrow}
\newcommand{\longto}{\longrightarrow}
\newcommand{\acts}{\curvearrowright}
\renewcommand{\tensor}{\otimes}
\newcommand{\consum}{\;\#\;}
\newcommand{\st}{\,\mid\,}
\newcommand{\sm}{\smallsetminus}

\newcommand{\braket}[1]{\langle #1\rangle}

\newcommand{\tGamma}{{\tilde{\Gamma}}}
\newcommand{\dGamma}{\partial\Gamma}
\newcommand{\vell}{{\vec{\ell}}}
\newcommand{\vk}{{\vec{k}}}
\newcommand{\vx}{{\vec{x}}}

\newcommand{\anchvert}{{\scriptsize\anchor}}

\def\app#1#2{\mathrel{\setbox0=\hbox{$#1\sim$}%
    \setbox2=\hbox{\rlap{\hbox{$#1\propto$}}\lower1.1\ht0\box0}%
    \raise0.25\ht2\box2}}

\def\api#1#2{\mathrel{%
    \setbox0=\hbox{$#1\sim$}%
    \setbox2=\hbox{\rlap{\hbox{$#1\in$}}\lower1.5\ht0\box0}%
    \raise0.35\ht2\box2}}

\newcommand{\defeq}{\stackrel{\text{def}}=}

\newenvironment{eq}[1]{\begin{equation} \label{#1}}
    {\end{equation}\ignorespacesafterend}
\newcommand{\eatline}{\vspace{-\baselineskip}}

\begin{document}
\title{Odd Vassiliev invariants vanish at four loops}
\author{Greg Kuperberg}
\email{greg@math.ucdavis.edu}
\thanks{Partly supported by NSF grant CCF-2317280.}
\affiliation{University of California, Davis}

\date{\today}

\begin{abstract}

\centerline{\textit{\normalsize \`A mes amis \`a l'Institut
    Fourier \`a Grenoble en 2010--2011}}
\vspace{\baselineskip}

The odd Vassiliev conjecture asserts that all Jacobi diagrams with an
odd number of legs vanish modulo the AS and IHX relations.  Equivalently,
the conjecture asserts that Vassiliev invariants cannot distinguish a knot
from its inverse. The conjecture is elementary for one 1 or 2 loops, and was
proven by Moskovich and Ohtsuki for 3 loops.  We establish the conjecture for
4 loops.  Our proof uses Jacobi diagrams and weight systems with polynomial
coefficients, commuting colored legs, and several vanishing criteria.
\end{abstract}

\maketitle

\section{Introduction}

An invariant function $I:\cK \to \Q$ on the set of isotopy classes of
oriented knots is \emph{finite type} or \emph{Vassiliev} if some $n+1$st
finite difference with respect to switching crossings vanishes identically.
A \emph{Jacobi diagram} is a unitrivalent graph, considered as a vector in
the vector space of such diagrams modulo the AS and IHX relations.  Connected
Jacobi diagrams are dual to primitive generators of the Hopf algebra of
Vassiliev invariants.  Jacobi diagrams are bigraded by the loop number $g$
and the leg number $\ell$, where the corresponding primitive Vassiliev
invariants have degree $g+\ell-1$.  See Bar-Natan \cite{Bar-Natan:vassiliev}.

The definition of a Vassiliev invariant uses both an orientation of space
and an orientation along the knot $K$.  Although they have no trouble
detecting a reflection of space, the question arises whether any Vassiliev
invariant can detect the more subtle inversion of the orientation along the
knot \cite{Trotter:noninvert}.  If Vassiliev invariants cannot detect knot
inversion, then they cannot distinguish certain prime, unoriented knots
\cite{K:invert}, but they might still distinguish all prime, atoroidal knots.

Inverting an oriented knot yields an involution on Vassiliev invariants and
Jacobi diagrams.  A Vassiliev invariant is called \emph{odd} (respectively
\emph{even}) when it is an eigenvector of this involution with eigenvalue
$-1$ (respectively $+1$).  Dually, every Jacobi diagram is an eigenvector,
with eigenvalue $(-1)^\ell$ when it has $\ell$ legs.  All known Vassiliev
invariants are even, which suggests that they all are.  Dually, there are
no odd Vassiliev invariants if and only if every Jacobi diagram with an
odd number of legs is annihilated by the AS and IHX relators.

Odd vanishing is elementary when a Jacobi diagram $\tGamma$ has 1 or 2
loops, and was established by Moskovich and Ohtsuki when $\tGamma$ has 3
loops \cite{MO:vanishing}.  (It also holds for 1 leg and any loop number
$g$ \cite[Prop.~4.3]{Vogel:diagrams}; and it holds when $g+\ell \le 13$
\cite{Kneissler:twelve}.)  In this article, we advance the Moskovich--Ohtsuki
result by 1.

\begin{theorem} Let $\cP_{g,\ell}$ be the rational vector space of
connected Jacobi diagrams with $g$ loops and $\ell$ legs, modulo the AS
and IHX relations.  Then $\cP_{g,\ell} = 0$ when $g \le 4$ and $\ell$ is odd.
\label{th:odd4} \end{theorem}

Moskovich--Ohtsuki interpreted the $\ell$ legs in a diagram $\tGamma$ with
$g$ loops as a polynomial of degree $\ell$ in $g$ variables.  Likewise, our
proof interprets the legs of a Jacobi diagram $\tGamma \in \cP_{g,\ell}$
as a polynomial $p \in S^\ell(H^1(\Gamma;\Q))$ in the cohomology of the
skeleton $\Gamma$ of $\tGamma$. We extend this interpretation to diagrams
with commuting, colored legs to obtain a vanishing lemma for diagrams
with bridges.  We also use weight polynomials in homology to obtain other
vanishing lemmas.

\acknowledgments

The author would like to thank Dror Bar-Natan, Christine Lescop, Jeff
Marshall-Milne Jean-Baptiste Meilhan, Andrew Putman, Dylan Thurston,
and Karen Vogtmann for useful discussions.

\section{Preliminaries}

Throughout this article, all linear algebra is over the rational numbers
$\Q$ unless otherwise specified. This includes coefficients for homology
and cohomology, so that we write
\[ H_n(X) \defeq H_n(X;\Q), \]
etc.

In some sections, we will use the concepts of the invariant subspace
and coinvariant quotient of a linear action $G \acts V$ of a group $G$
on a vector space $V$:
\begin{align*}
\Inv_G(V) &\defeq \{v \in V \st gv = v,\;\forall g \in G\} \\
    \Coinv_G(V) &\defeq V/\braket{gv-v,\;\forall g \in G}.
\end{align*}
Thus $\Inv_G(V)$ is the largest subspace of $V$ fixed by $G$, while
$\Coinv_G(V)$ is the largest quotient of $V$ fixed by $G$.

\subsection{Jacobi diagrams}

A \emph{Jacobi diagram} is a finite, unitrivalent, topological graph with
a cyclic ordering of the edges incident to each trivalent vertex.  In this
article, we will only consider connected Jacobi diagrams.  The univalent
vertices are called \emph{legs}, and we adopt the convention that a diagram
with legs is denoted $\tGamma$, while a trivalent diagram with no legs
is denoted $\Gamma$.  In particular, if we remove the legs of $\tGamma$,
we obtain its trivalent \emph{skeleton} $\Gamma$.  We will sometimes
consider $\Gamma$ without yet orienting the vertices, since (by the AS
relation below) different choices are equivalent up to sign.

We say that either $\tGamma$ or $\Gamma$ has $g$ \emph{loops} when the
homology space
\[ H_1(\tGamma) \cong H_1(\Gamma) \cong \Q^g \]
is $g$-dimensional.  In this article, we assume that $g > 0$.  Since $\Gamma$
is a topological graph, it can be a circle (when $g=1$). It can also have
self-loops (unigons) or double edges (digons); and it can be a triple edge
(a theta graph).

Our main object of study is the Jacobi space
\[ \cP_{g,\ell} \defeq \cJ_{g,\ell}/\braket{\AS,\IHX}. \]
By definition, it is the vector space $\cJ_{g,\ell}$ of formal linear
combinations of Jacobi diagrams with $g$ loops and $\ell$ legs, quotiented
by the AS (anti-symmetry) and IHX (Jacobi identity) relators:
\begin{eq}{e:asihx} \AS = \begin{tp}[style=basic]
    \draw (0,0) -- (0,1) (-.8,-.8) -- (0,0) -- (.8,-.8); \end{tp} -
\begin{tp}[style=basic] \draw (0,.4) -- (0,1) (1,-.8)
    .. controls (.6,-.6) and (-1,-.2) .. (0,.4);
\draw[style=overcross] (0,.4)
    .. controls (1,-.2) and (-.6,-.6) .. (-1,-.8); \end{tp} \qquad
\IHX = \begin{tp}[style=basic] \draw (-.6,1.1) -- (0,.6) -- (.6,1.1)
    (-.6,-1.1) -- (0,-.6) -- (.6,-1.1) (0,.6) -- (0,-.6); \end{tp} -
\begin{tp}[style=basic] \draw (1.1,-.6) -- (.6,0) -- (1.1,.6)
    (-1.1,-.6) -- (-.6,0) -- (-1.1,.6) (.6,0) -- (-.6,0); \end{tp} +
\begin{tp}[style=basic] \draw (-1.2,-1) -- (.8,1);
    \draw[style=overcross] (-.8,1) -- (1.2,-1);
    \draw (-.6,-.4) -- (.6,-.4); \end{tp} \end{eq}
Jacobi diagrams can thus also be interpreted as Jacobi space elements
$\tGamma \in \cP_{g,\ell}$.  Since the AS relation implies that $\tGamma =
0$ if it has two legs that meet at the same trivalent vertex, we can assume
that the legs of $\tGamma$ are attached to distinct trivalent vertices.
We also abbreviate the Jacobi spaces
\[ \cP_g \defeq \cP_{g,0} \]
of trivalent diagrams, which we will generalize to Jacobi spaces
$\cP_g(V)$ with twisted coefficients in \Sec{ss:twisted}.

A \emph{weight system} is a linear function $w:\cP_{g,\ell} \to \Q$
and thus a vector $w \in \cP_{g,\ell}^*$ in the dual vector space
$\cP_{g,\ell}^*$.  Whereas the Jacobi space $\cP_{g,\ell}$ is defined by
quotienting $\cJ_{g,\ell}$ by the AS and IHX relators, weight systems lie
in a subspace $\cP_{g,\ell}^* \subseteq \cJ_{g,\ell}^*$ of functions on
the set of Jacobi diagrams, namely those that the AS and IHX relations.

We will also use diagrams with colored legs (more precisely, colored socks
on the feet of the legs). For example:
\[ \begin{tp}[style=basic]
\draw (0,0) circle (2);
\foreach \th in {90,210,330} { \draw (0,0) -- (\th:2); }
\foreach \th/\c in {10/medgreen,50/medgreen,135/medred,165/medblue} {
    \draw (\th:2) -- (\th:2.8); \fill[\c] (\th:2.9) circle (.25); }
\end{tp} \]
When the legs are colored, we add a relator CL to make legs commute:
\[ \CL =\; \begin{tp}[style=basic,shift={(0,-.4)}]
\draw (-1.5,0) -- (1.5,0) (-.6,0) -- (-.6,.8) (.6,0) -- (.6,.8);
\fill[medgreen] (-.6,.9) circle (.25);
\fill[medblue] (.6,.9) circle (.25);\end{tp}
\;-\;\begin{tp}[style=basic,shift={(0,-.4)}]
\draw (-1.5,0) -- (1.5,0) (-.6,0) -- (-.6,.8) (.6,0) -- (.6,.8);
\fill[medblue] (-.6,.9) circle (.25);
\fill[medgreen] (.6,.9) circle (.25);\end{tp} \]

Given any integer partition
\[ \vell = (\ell_1,\ell_2,\ldots,\ell_n)
    \qquad \ell = \ell_1 + \ell_2 + \cdots + \ell_n \]
of the total leg number $\ell$, we can define a generalized Jacobi space
$\cP_{g,\vell}$ by giving $\ell_j$ of the legs the $j$th color.  We say
that $\vell$ is odd or even respectively when $\ell$ is odd or even.
More formally, let $\cJ_{g,\vell}$ be the space of linear combinations
of Jacobi diagrams with $g$ loops and legs with color population $\vell$.
We then define
\[ \cP_{g,\vell} \defeq \cJ_{g,\vell}/\braket{\AS,\IHX,\CL}. \]

We will also consider Jacobi diagrams with legs ending in \emph{anchors}
that are numbered and do not commute with each other, nor with
ordinary legs. For example:
\[ \begin{tp}[style=basic]
\draw (-3.25,0) -- (-1.25,0) (0,0) circle (1.25) (1.25,0) -- (3.75,0);
\foreach \th in {45,90,135,270} { \draw (\th:1.25) -- (\th:2.05); }
\draw (2.5,0) -- (2.5,.8);
\draw (-3.05,0) node[left] {\anchvert$_1$}
    (3.8,0) node[right] {\anchvert$_2$};
\end{tp} \]
Except in this example, we will omit the anchor numbers.  The diagrams
with $g$ loops, $a$ distinguished anchors, and $\ell$ legs yield an
\emph{anchored} Jacobi space $\cP_{g,a,\ell}$.  If a Jacobi diagram
$\tGamma$ has anchors, then we include those anchors in its skeleton
$\Gamma$ and only remove the other legs.  We assume that $g+a > 0$ for a
uniform treatment of Jacobi spaces and skeleta.

Of course, we can generalize $\cP_{g,a,\ell}$ to $\cP_{g,a,\vell}$ with
colored legs, and we can restrict to the special case where there are no
anchors after all:
\[ \cP_{g,\vell} = \cP_{g,0,\vell}. \]

As discussed in \Sec{ss:consum}, the point of the anchors is to connect
diagrams to each other.  Thus the AS, IHX, and CL relators above should
technically have anchors.

\begin{remark} Here we indulge in a modest brag concerning terminology.
Google Scholar indicates that the author and Dylan Thurston were the
first to publish the now-standard term ``Jacobi diagram'' \cite{K:paste}.
Our motivation was that terms such as ``Feynman diagram'' or ``graph''
are too general; while ``Chern-Simons Feynman diagram'' is verbose, and
``Chinese character diagram'' is both verbose and a strange metaphor.
Of course, the term ``Jacobi diagram'' refers to the fact that the IHX
relation is also the Jacobi identity in the definition of a Lie algebra.
\end{remark}

\subsection{Twisted coefficients}
\label{ss:twisted}

Let $\Gamma$ be a graph with $g$ loops, and which need not be trivalent
but has no legs.  We can decorate $\Gamma$ with a \emph{homotopy framing},
meaning that we choose a homotopy equivalence
\[ f:\Gamma \stackrel{\sim}{\longto} \Phi_g \]
to a fixed graph $\Phi_g$ with $g$ loops. For definiteness, we let $\Phi_g$
be a bouquet of $g$ circles.  If $\Gamma_1$, $\Gamma_2$, and $\Gamma_3$
are trivalent and participate in an IHX relator \eqref{e:asihx}, then the
relator yields canonical homotopy equivalences $\Gamma_1 \sim \Gamma_2
\sim \Gamma_3$, because all three are canonically homotopy equivalent to
the graph $\Gamma_4$ given in each case by contracting the middle
edge to a point:
\begin{eq}{e:gamma4}
\begin{tp}[style=basic] \draw (-.6,1.1) -- (0,.6) -- (.6,1.1)
    (-.6,-1.1) -- (0,-.6) -- (.6,-1.1) (0,.6) -- (0,-.6);
    \draw (0,-2.5) node {$\Gamma_1$}; \end{tp} \;\sim\;
\begin{tp}[style=basic] \draw (1.1,-.6) -- (.6,0) -- (1.1,.6)
    (-1.1,-.6) -- (-.6,0) -- (-1.1,.6) (.6,0) -- (-.6,0);
    \draw (0,-2.5) node {$\Gamma_2$}; \end{tp} \;\sim\;
\begin{tp}[style=basic] \draw (-1.2,-1) -- (.8,1);
    \draw[style=overcross] (-.8,1) -- (1.2,-1);
    \draw (-.6,-.4) -- (.6,-.4);
    \draw (0,-2.5) node {$\Gamma_3$}; \end{tp} \;\sim\;
\begin{tp}[style=basic] \draw (-1,-1) -- (1,1) (-1,1) -- (1,-1);
    \fill (0,0) circle (.25);
    \draw (0,-2.5) node {$\Gamma_4$}; \end{tp} \end{eq}

Thus the IHX relator makes sense for homotopy-framed Jacobi diagrams.
The AS relator does too, because the two terms are given by the same graph
$\Gamma$ except with opposite cyclic orderings at one vertex.  We will
use this enhancement of the relators to define generalized Jacobi spaces
with local coefficients, in a manner analogous to local coefficients in
algebraic topology.

Recall that for any discrete group $G$, the (unpointed) homotopy automorphism
group of its classifying space $K(G,1)$ is its outer automorphism group:
\[ \hAut(K(G,1)) \cong \Out(G). \]
Recall also that $\Phi_g$ (or any graph with $g$ loops) is a classifying
space of a free group, hence:
\begin{eq}{e:out} \Phi_g \sim K(F_g,1)
    \qquad \hAut(\Phi_g) \cong \Out(F_g). \end{eq}
If $\Out(F_g) \acts V$ is a linear representation of $\Out(F_g)$ (over
$\Q$ as usual for us), then the IHX and AS relators for homotopy-framed
Jacobi diagrams yield a twisted Jacobi space $\cP_g(V)$.  Since each
Jacobi diagram $\Gamma$ is given a homotopy framing $\Gamma \sim \Phi_g$
to define $\cP_g(V)$, we can also describe $V$ as a linear representation
$\hAut(\Gamma) \acts V$.

More formally, let $\cJ_g(V)$ be the vector space of formal linear
combinations of homotopy-framed trivalent diagrams with $g$ loops and
with coefficients in $V$.  There is an action $\Out(F_g) \acts \cJ_g(V)$
in which $\Out(F_g)$ both acts on the framing of any given $\Gamma$ and
acts on $V$ as a representation.  Then
\[ \cP_g(V) \defeq \Coinv_{\Out(F_g)}(\cJ_g(V))/\braket{\AS,\IHX}. \]

\begin{example} Let
\[ V \defeq S^\ell(H^1(\Gamma)) \cong S^\ell(\Q^g) \]
be the $\ell$th symmetric power of the first cohomology of $\Gamma$, which
has a natural action of $\hAut(\Phi_g)$. In this case, \Lem{l:sympow}
below states that the legged Jacobi space $\cP_{g,\ell}$ is isomorphic to
the twisted Jacobi space $\cP_g(S^\ell(H^1(\Gamma)))$.
\end{example}

\begin{remark} Every coefficient space $V$ that we use in this article
is made from the (co)homology of $\Gamma$. Thus we can use a simpler
\emph{homology framing} of $\Gamma$, meaning an integral basis of
$H_1(\Gamma;\Z)$.  We can define any such $V$ as a linear representation
of the general linear group $\GL(g,\Z)$ instead of its extension $\Out(F_g)$.
\end{remark}

We can generalize the model graph $\Phi_g$ to a graph $\Phi_{g,a}$ which is
a bouquet of $g$ loops and a set $A$ of $a \ge 1$ anchored legs; for example:
\[ \Phi_{4,3} =\;\; \begin{tp}[style=basic]
\foreach \th in {0,60,120,180} { \draw (0,0)
    .. controls ({\th-30}:1) and ($(\th:3)-({\th+90}:1.35)$) .. (\th:3)
    .. controls ($(\th:3)+({\th+90}:1.35)$) and (\th+30:1) .. (0,0); }
\foreach \th in {235,270,305} {
\draw (0,0) -- (\th:2.5); \draw (\th:3.15) node {\anchvert}; }
\fill (0,0) circle (.25);
\end{tp} \]
We can likewise generalize \equ{e:out} to the relative homotopy automorphism
group $\hAut(\Phi_{g,a},A)$; by definition the group of homotopy equivalences
that fix $A$ as follows \cite{BF:out}.  For any discrete group $G$, let
$K(G,1)$ be a classifying space, and let $A \subseteq K(G,1)$ be a subset
with $a$ points.  As explained by Bestvina and
Feighn \cite[Sec.~2.5]{BF:out} in the special case $G = F_g$,
we always have:
\[ \hAut(K(G,1),A) \cong \Aut(G) \ltimes G^{a-1}. \]
To state the special case explicitly,
\[ \hAut(\Phi_{g,a},A) \cong \Aut(F_g) \ltimes F_g^{a-1}. \]
A linear representation $\hAut(\Phi_{g,a},A) \acts V$ likewise yields an
anchored Jacobi space $\cP_{g,a}(V)$ with twisted coefficients.  In this
article, this $V$ is always a polynomial or tensor space over the relative
cohomology $H^1(\Phi_{g,a},A)$.

Finally, given a linear representation $\hAut(\Phi_{g,a},A) \acts V$,
we can dually consider weight systems $w \in \cP_g^*(V)$.  By definition,
these are functions on homotopy-framed Jacobi diagrams that take values in
$V$, that satisfy the homotopy-framed AS and IHX relations, and that are
$\hAut(\Phi_{g,a},A)$-invariant.  Note that when we dualize from Jacobi
diagrams to weight systems, we must also dualize the coefficients:
\begin{eq}{e:dual} \cP_g(V)^* \cong \cP_g^*(V^*). \end{eq}
This is analogous to the duality between the homology and cohomology of
a topological space $X$ with coefficients in a flat vector bundle $V$
over any field $F$:
\[ H^n(X;V^*) \cong H_n(X;V)^*. \]
As discussed in the next subsection, it is more than an analogy.

\begin{figure*}[t] \begin{center}
\includegraphics[width=4in]{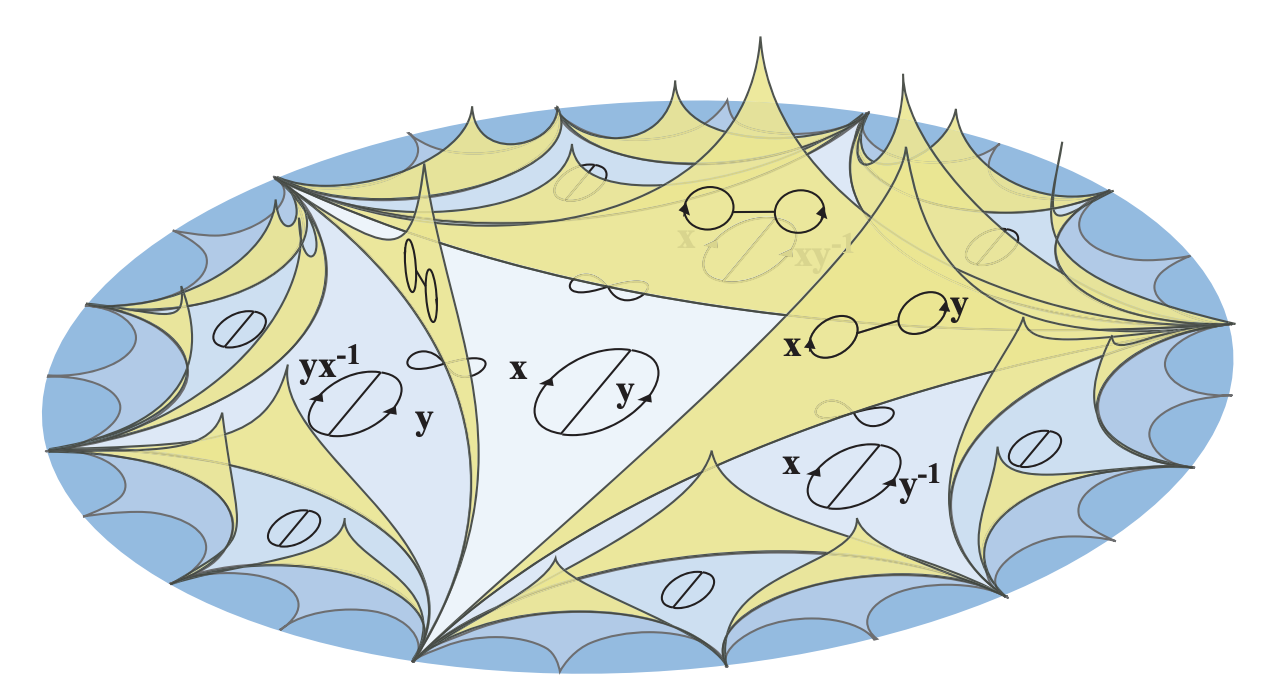}
\end{center}
\caption{Outer space $\cO_2$ in its simplicial closure $\bcO_2$ for two loops
    (figure taken from Vogtmann \cite{Vogtmann:outer}).}
\label{f:outer} \end{figure*}

\subsection{Outer space}
\label{ss:outer}

Although this subsection is optional for our results, we include it
to help interpret Jacobi diagrams and spaces with twisted coefficients.
Homotopy framings were also used by Culler and Vogtmann to make \emph{Outer
space} $\cO_g$ \cite{CV:moduli}; see \Fig{f:outer}.  For each $g \ge
2$, Outer space $\cO_g$ is defined as a moduli space of metric graphs
$\Gamma$ with $g$ loops and a homotopy framing $\Gamma \sim \Phi_g$.
$\cO_g$ also lies in a simplicial complex $\bcO_g$, its \emph{simplicial
closure}, meaning that $\cO_g$ is dense in $\bcO_g$ and is the complement
of a simplicial subcomplex $\cI_g \subseteq \bcO_g$ of ideal points.
Both $\cO_g$ and $\bcO_g$ are also contractible and have a canonical
simplicial action of $\Out(F_g)$.

As established by Conant and Vogtmann \cite[Prop.~27]{CV:kontsevich}, the
Jacobi space $\cP_g$ is isomorphic to a top-dimensional, equivariant,
relative, twisted cohomology of Outer space:
\begin{eq}{e:equiv} \cP_g \cong H^{3g-4}_G(\bcO_g,\cI_g;L).
\end{eq}
It is easier to establish \equ{e:equiv} with a simplification of
Borel equivariant cohomology, namely the homology of the complex of
$\Out(F_g)$-invariant cochains.  Since we are working over $\Q$, and since
$\Out(F_g) \acts \cO_g$ is a properly discontinuous action of a discrete
group, this simplification yields the same answer \cite[Sec.~3.1]{CP:codim}.
The coefficient system $L \cong L^*$ in \equ{e:equiv} is a line bundle on
$\cO_g$ induced by the determinant representation of $\GL(g,\Z)$, extended
to a representation of $\Out(F_g)$.  It doesn't extend as a line bundle
to the relative locus $\cI_g \subseteq \bcO_g$, but it doesn't need to.
(Equivalently, it extends to the zero sheaf on $\cI_g$.)  In the isomorphism
with $\cP_g$, $L$ absorbs the AS relator and orientations of simplices in
simplicial homology.

As a generalization of the Conant--Vogtmann isomorphism, the twisted Jacobi
space $\cP_g(V)$ that we use is isomorphic to a cohomology space twisted
by both $L$ and $V$:
\[ \cP_g(V) \cong H^{3n-4}_{\Out(F_g)}(\bcO_g,\cI_g;L \tensor V). \]

Outer space $\cO_{g,a}$ generalizes to a moduli space $\cO_{g,a}$ of metric
graphs $\Gamma$ with $g$ loops and $a$ anchors, using homotopy equivalences
$\Gamma \sim \Phi_{g,a}$ with the model graph $\Phi_{g,a}$ that extend
the label-preserving bijection of the anchors \cite[Sec.~2.5]{BF:out}.
Given a linear representation $\hAut(\Phi_{g,a},A) \acts V$, we again get
an isomorphism with a top cohomology space:
\[ \cP_{g,a}(V) \cong H^{3n+a-4}_{\hAut(\Phi_{g,a},A)}
    (\bcO_{g,a},\cI_{g,a};L \tensor V). \]
Just as Outer space $\cO_g$ for diagrams with $g$ loops is analogous
to Teichm\"uller space $\cT_g$ for surfaces of genus $g$, Outer space
$\cO_{g,a}$ for anchored diagrams is analogous to Teichm\"uller space
$\cT_{g,a}$ for marked surfaces.

\section{Isomorphisms and maps}
\label{s:isomaps}

\subsection{Legs as twisted coefficients}

For any vector space $V$, let
\[ S^\vell(V) \defeq S^{\ell_1}(V) \tensor S^{\ell_2}(V) \tensor \cdots
    \tensor S^{\ell_n}(V) \]
denote the indicated tensor product of symmetric powers.

\begin{lemma} For every $g$, $\vell$, and $a$, there is a canonical
isomorphism
\[ \cP_{g,a,\vell} \cong \cP_{g,a}(S^\vell(H^1(\Gamma,\dGamma))), \]
where $\dGamma$ is the (possibly empty) set of anchors of $\Gamma$.
\label{l:sympow} \end{lemma}


This most important special case of \Lem{l:sympow} for us is
\[ \cP_{g,\ell} \cong \cP_g(S^\ell(H^1(\Gamma))), \]
but we will also rely on some other cases.  This special case is a more
abstract version of polynomial calculations in Moskovich--Ohtsuki
\cite{MO:vanishing}.

\begin{proof} We first build a local Jacobi space $\cL_\vell(\Gamma)$
for a single skeleton $\Gamma$, where the IHX and AS relators are only
used to move around the (unanchored) legs.  The edges of $\Gamma$ make it a
1-dimensional, finite cell complex except when $\Gamma$ is a circle, and we
will argue the lemma using cellular homology.  The lemma is straightforward
in the special case that $\Gamma$ is a circle.

We first establish an isomorphism
\begin{eq}{e:cohom} \cL_1(\Gamma) \cong H^1(\Gamma,\dGamma) \end{eq}
by converting a leg to an orientation of its underlying edge using the
cyclic ordering at its attachment vertex, as in this example:
\[ \begin{tp}[style=basic]
\foreach \th in {90,210,330} { \draw (0,0) -- (\th:2); }
\draw (210:2) arc (210:450:2);
\draw (270:2) -- (270:3.5); \draw (270:4.15) node {\anchvert};
\draw[medcyan,ultra thick] (90:2) arc (90:210:2);
\draw[medcyan,ultra thick] (150:2) -- (150:2.8);
\draw[<->] (3.5,-1) -- (6.5,-1); \tikzset{shift={(10,0)}};
\foreach \th in {90,210,330} { \draw (0,0) -- (\th:2); }
\draw (210:2) arc (210:450:2);
\draw (270:2) -- (270:3.5); \draw (270:4.15) node {\anchvert};
\draw[medcyan,ultra thick] (90:2) arc (90:210:2);
\draw[medcyan,ultra thick,->] (150:2) arc (150:145:2); \end{tp} \]
Thus a leg on any one edge becomes a 1-cochain; since $\Gamma$ is
1-dimensional, it is a 1-cocycle.  The AS relator at a leg matches the
antisymmetry of a cochain for a cell, while an IHX relator at a leg has
another participating vertex $v \in \Gamma$ and matches the coboundary
$\delta v$.  Thus \equ{e:cohom} holds because the two sides can be given
matching definitions.

We extend \equ{e:cohom} with a pair of isomorphisms
\begin{eq}{e:local} \cL_\vell(\Gamma) \cong S^\vell(\cL_1(\Gamma))
    \cong S^\vell(H^1(\Gamma,\dGamma)). \end{eq}
The second isomorphism is immediate from \eqref{e:cohom}.  The first
isomorphism is easier to discuss in the special case
$\cL_\ell(\Gamma)$ with uncolored legs.  In this
case, we employ a fact from abstract algebra: If $V$ is a vector
space with subspace $W \subseteq V$, then there is an algebra isomorphism
\[ S^*(V/W) \cong S^*(V)/\braket{W}. \]
In words, if we first quotient $V$ by $W$ as a subspace and then form the
symmetric algebra $S^*(V/W)$, we get the same thing as if we first form
the symmetric algebra $S^*(V)$ and then quotient by the ideal generated
by $W$.  In context, $V$ is the space of formal linear combinations of
one leg attached to $\Gamma$, while $W$ is the subspace spanned by the
AS and IHX relators.  Then $S^*(V)/\braket{W}$ matches the definition
of $\cL_\ell(\Gamma)$, while $S^*(V/W)$ matches the definition of
$S^\ell(\cL_1(\Gamma))$.

We can similarly argue the first isomorphism in \eqref{e:local} in full
generality using tensor algebras instead of symmetric algebras.

To finish the proof, we want to pass from the characterization of
$\cL_\vell(\Gamma)$ to the lemma's characterization of $\cP_{g,a,\vell}$.
To simplify the formulas, we assume that $a=0$; the case with anchors is
argued the same way using relative cohomology.  Consider what happens when
we apply a finite sequence of AS and IHX relators to a diagram $\tGamma$
to either $\cP_{g,\vell}$ or $\cP_g(S^\vell(H^1(\Gamma)))$.  There will be
a subsequence of relators that change the skeleton $\Gamma$, interspersed
with subsequences that only move the legs of $\tGamma$.  As long as only the
legs move, the value $\tGamma \in S^\vell(H^1(\Gamma)))$ does not change.
If we apply an AS relator to the skeleton $\Gamma$, then this also negates
$\tGamma$, which matches the definition of $\cP_g(S^\vell(H^1(\Gamma)))$.
If we apply an IHX relator to $\tGamma$ that changes $\Gamma$, then by
the homotopy equivalences \eqref{e:gamma4}, this matches the corresponding
IHX relator for $\cP_g(S^\vell(H^1(\Gamma)))$.

Finally, there is the more subtle case when we want to apply an IHX relator
at an edge $e \in \Gamma$ as an equivalence in $\cP_g(S^\vell(V))$,
but there are legs on $e$ in $\tGamma$.  Again using \eqref{e:gamma4},
the homotopy framing of each $\Gamma_k$ induces specific isomorphisms
\[ H^1(\Gamma_1) \cong H^1(\Gamma_2)
    \cong H^1(\Gamma_3) \cong H^1(\Gamma_4). \]
If we evacuate the legs of $\tGamma$ from $e \in \Gamma = \Gamma_1$ using
the IHX relator for legs, then we obtain a common element $\tGamma \in
S^\ell(H^1(\Gamma_k))$ for all $k$.  It is the same element regardless
of how the legs are evacuated.  We can then apply an IHX relator in
$\cP_{g,\ell}$ to match the one in $\cP_g(S^\ell(H^1(\Gamma)))$.
\end{proof}

\subsection{Connected sums}
\label{ss:consum}

When two Jacobi diagram $\tGamma_1$ and $\tGamma_2$ each have exactly one
leg with the same color, then we can form a connected sum diagram $\tGamma =
\tGamma_1\consum\tGamma_2$ by attaching those two legs to each other:
\[ \begin{tp}[style=basic,scale=.9]
\draw (0,0) circle (2);
\foreach \th in {60,180,300} { \draw (0,0) -- (\th:2); }
\foreach \th/\c in {100/medblue,140/medblue,240/medgreen,0/medred} {
    \draw (\th:2) -- (\th:2.8); \fill[\c] (\th:2.9) circle (.25); }
\end{tp} \;\;\#\;\; \begin{tp}[style=basic,scale=.9]
\draw (0,0) ellipse (2 and 1.414) (0,1.414) -- (0,-1.414);
\draw (-2,0) -- (-2.8,0); \fill[medred] (-2.9,0) circle (.25);
\draw (0,0) -- (.8,0); \fill[medblue] (.9,0) circle (.25);
\draw ({sqrt(8/3)},{-sqrt(2/3)}) -- ++({.8/sqrt(2)},{-.8/sqrt(2)});
\draw ({sqrt(8/3)},{sqrt(2/3)}) -- ++({.8/sqrt(2)},{.8/sqrt(2)});
\fill[medblue] ({sqrt(8/3)},{-sqrt(2/3)})
    ++({.9/sqrt(2)},{-.9/sqrt(2)}) circle (.25);
\fill[medgreen] ({sqrt(8/3)},{sqrt(2/3)})
    ++({.9/sqrt(2)},{.9/sqrt(2)}) circle (.25);
\end{tp} \;\;=\;\; \begin{tp}[style=basic,scale=.9]
\draw (0,0) circle (2);
\foreach \th in {60,180,300} { \draw (0,0) -- (\th:2); }
\foreach \th/\c in {100/medblue,140/medblue,240/medgreen} {
    \draw (\th:2) -- (\th:2.8); \fill[\c] (\th:2.9) circle (.25); }
\draw (2,0) -- (3,0);
\tikzset{shift={(5,0)}};
\draw (0,0) ellipse (2 and 1.414) (0,1.414) -- (0,-1.414);
\draw (0,0) -- (.8,0); \fill[medblue] (.9,0) circle (.25);
\draw ({sqrt(8/3)},{-sqrt(2/3)}) -- ++({.8/sqrt(2)},{-.8/sqrt(2)});
\draw ({sqrt(8/3)},{sqrt(2/3)}) -- ++({.8/sqrt(2)},{.8/sqrt(2)});
\fill[medblue] ({sqrt(8/3)},{-sqrt(2/3)})
    ++({.9/sqrt(2)},{-.9/sqrt(2)}) circle (.25);
\fill[medgreen] ({sqrt(8/3)},{sqrt(2/3)})
    ++({.9/sqrt(2)},{.9/sqrt(2)}) circle (.25); \end{tp} \]

\begin{lemma} Connected sum yields a bilinear map
\[ \#:\cP_{g,(1,\vell)} \times \cP_{h,(1,\vk)}
    \to \cP_{g+h,\vell+\vk}. \]
\label{l:consum} \end{lemma} \eatline \eatline

\begin{proof} Applying an AS or IHX relator to either $\tGamma_1$ or
$\tGamma_2$ commutes with the formation of $\tGamma = \tGamma_1 \consum
\tGamma_2$; this part of the lemma is elementary.  The same remark holds if
we swap two adjacent, colored legs in either $\tGamma_1$ or $\tGamma_2$,
other than the one used in the connected sum.

The remaining question is the CL relator, \ie, what happens if we swap the
leg used in the connected sum with an adjacent leg in either $\tGamma_1$
or $\tGamma_2$.  We claim that $\tGamma_1$ as a summand of $\tGamma$
commutes with an adjacent leg $e$ of $\tGamma_2$ (or vice versa), as in
the following example:
\[ \begin{tp}[style=basic,shift={(0,-1.5)}]
\draw (0,0) circle (2);
\foreach \th in {60,180,300} { \draw (0,0) -- (\th:2); }
\draw (120:2) -- (120:3.4) (120:4.2) circle (.8);
\foreach \th/\c in {50/medgreen,120/medred,190/medblue} {
    \draw (120:4.2)++(\th:.8) -- ++(\th:.8);
    \fill[\c] (120:4.2)++(\th:1.7) circle (.25); }
\foreach \th/\c in {150/medblue,0/medred} {
    \draw (\th:2) -- (\th:2.8); \fill[\c] (\th:2.9) circle (.25); }
\draw (120:4.2)++(-3,0) node {$\tGamma_1$} (-3.7,0) node {$\tGamma_2$};
\draw (150:1.5) node {$e$};
\end{tp} \;\;\;= \begin{tp}[style=basic,shift={(0,-1.5)}]
\draw (0,0) circle (2);
\foreach \th in {60,180,300} { \draw (0,0) -- (\th:2); }
\draw (120:2) -- (120:3.4) (120:4.2) circle (.8);
\foreach \th/\c in {50/medgreen,120/medred,190/medblue} {
    \draw (120:4.2)++(\th:.8) -- ++(\th:.8);
    \fill[\c] (120:4.2)++(\th:1.7) circle (.25); }
\foreach \th/\c in {90/medblue,0/medred} {
    \draw (\th:2) -- (\th:2.8); \fill[\c] (\th:2.9) circle (.25); }
\draw (120:4.2)++(-3,0) node {$\tGamma_1$} (-3.7,0) node {$\tGamma_2$};
\draw (90:1.5) node {$e$}; \end{tp} \]
As the example illustrates, the two positions for $e$ are cohomologous in
the skeleton $\Gamma$. Thus the summand $\tGamma_1$ and the leg $e$ commute.
\end{proof}

We will also rely on an analogous connected sum operation with anchors.
If $\tGamma_1$ and $\tGamma_2$ have anchors, then we can make a connected
sum $\tGamma_1 \;\#_b\; \tGamma_2$ for any partial bijection $b$ between
their anchors.  \Lem{l:consum} also holds for this type of connected sum,
except that it is elementary because anchors do not have a CL relator.

\subsection{Homology}
\label{ss:homology}

Let $\Gamma$ be a trivalent diagram.  Since we are working over $\Q$,
and since $\Gamma$ is a finite graph, the universal coefficient theorem yields
a canonical duality between polynomials in homology and polynomials in
cohomology:
\[ S^\ell(H_1(\Gamma)) \cong S^\ell(H^1(\Gamma))^*. \]
By \equ{e:dual}, we can describe a weight system $w \in \cP_{g,\ell}^*$
as a \emph{weight polynomial} $w_\Gamma$ for each skeleton $\Gamma$,
using the isomorphism
\[ \cP_{g,\ell}^* \cong \cP_g^*(S^\ell(H_1(\Gamma))). \]
We can say either that $w_\Gamma$ is a polynomial in the homology
$H_1(\Gamma)$, or a polynomial function on the cohomology $H^1(\Gamma)$.
Each weight polynomial $w_\Gamma$ must transform to itself under
automorphisms of $\Gamma$, in some cases with a sign correction given by
the AS relation.  In addition, the polynomials $w_\Gamma$ must satisfy
the IHX relation.

To help interpret calculations, it is useful to use dual bases.
For each $\Gamma$, we can choose a spanning tree $T \subseteq \Gamma$;
and for each edge $e_k \in \Gamma \sm T$, we can choose an orientation.
Then these oriented edges represent a basis $\{e_k\} \subseteq H^1(\Gamma)$.
Each $e_k$ also completes uniquely to a cycle $x_k$ using edges in $T$;
and these cycles represent the dual basis $\{x_k\} \subseteq H_1(\Gamma)$.
Each pair of monomials
\[ e_{k_1} e_{k_2} \cdots e_{k_\ell} \in S^\ell(H^1(\Gamma)) \qquad
x_{k_1} x_{k_2} \cdots x_{k_\ell} \in S^\ell(H_1(\Gamma)) \]
are then also dual up to a multinomial coefficient factor.

\begin{example} If $\Gamma$ is a tetrahedron, then we can let the spanning
tree $T$ be the three edges in the middle as indicated in yellow, and form
the indicated dual bases of $H_1(\Gamma)$ and $H^1(\Gamma)$:
\begin{eq}{e:tetra} \Gamma =\begin{tp}[style=basic,very thick]
\foreach \th in {90,210,330} {
    \draw[medyel, ultra thick] (0,0) -- (\th:3.45); }
\draw (0,0) circle (3.45);
\foreach \th in {90,210,330} {
    \draw[medcyan,->,rounded corners=3pt] ({\th+65}:3)
        arc ({\th+65}:{\th+120-atan(.15)}:3) -- ({\th+60}:{.9/sqrt(3)})
        -- ({\th+atan(.15)}:3) arc ({\th+atan(.15)}:{\th+55}:3); }
\foreach \th in {90,210,330} {
    \draw[medmag,->] ({\th+25}:3.9) arc ({\th+25}:{\th+95}:3.9); }
\foreach \th/\n in {150/1,270/3,30/2} {
    \draw (\th:1.95) node {$x_\n$} (\th:4.7) node {$e_\n$}; }
\end{tp} \end{eq}
Each element $f \in S_4$ of the automorphism group of the tetrahedron
graph $\Gamma$ reverses an even number of vertices.  It also permutes the
4 elements
\[x_1, x_2, x_3, x_4 \in H_1(\Gamma) \qquad x_4 \defeq -x_1-x_2-x_3,\]
together with a scalar action of its sign $(-1)^f$ on $H_1(\Gamma)$.
If a weight polynomial $w_\Gamma(\vx)$ participates in a weight system $w
\in \cP_{4,\ell}$, then it must be invariant under this $S_4$ action.

For example, the degree $\ell=4$ polynomial
\[ w_\Gamma(\vx) = -x_1x_2x_3x_4
    = x_1^2x_2x_3 + x_1x_2^2x_3 + x_1x_2x_3^2 \]
has the required symmetry.  By a derivation of Moskovich and Ohtsuki
\cite[Sec.~3.4]{MO:vanishing}, this $w_\Gamma$ (or any $S_4$-invariant
polynomial with $\ell$ even) extends uniquely to a valid weight system $w
\in \cP_{3,\ell}^*$ given by a polynomial for each skeleton $\Gamma$.
\end{example}

\section{Vanishing lemmas}

In this subsection, we establish several vanishing lemmas for Jacobi
diagrams.  Taken together, the lemmas yield a proof of \Thm{th:odd4}
with little remaining calculation.

\begin{lemma}[Symmetry vanishing] Let:
\begin{itemize}
\item $\tGamma$ be a Jacobi diagram with $g$ loops and an odd leg color
population $\vell$.
\item $f$ be an involution of the skeleton $\Gamma$ that fixes $n$
vertices while reversing their orientation.
\item $V \subseteq H^1(\Gamma)$ be the eigenspace of $f$ with
eigenvalue $(-1)^n$.
\end{itemize}
If the natural map
\[ S^\vell(H^1(\Gamma)) \cong \cL_\vell(\Gamma) \to \cP_{g,\vell} \]
factors as
\[ S^\vell(H^1(\Gamma)) \onto S^\vell(H^1(\Gamma)/V) \to \cP_{g,\vell}, \]
then it vanishes identically, \ie, every $\tGamma = 0 \in \cP_{g,\vell}$.
\label{l:symvan} \end{lemma}

Although we opted to state all of our vanishing lemmas in terms of Jacobi
diagrams, we will apply \Lem{l:symvan} dually to weight polynomials.
Given $\Gamma$, let $W \subseteq H_1(\Gamma)$ be the homology eigenspace
of the involution $f$ with eigenvalue $(-1)^{n+1}$.  If every $w \in
\cP_{g,\vell}$ restricts to a weight polynomial
\[ w_\Gamma \in S^\ell(W) \subseteq S^\ell(H_1(\Gamma)), \]
then (by symmetry) every such $w_\Gamma$ vanishes.

\begin{proof} The automorphism $f$ also acts by the scalar $(-1)^{n+1}$ on
the quotient $H^1(\Gamma)/V$.  Since $\vell$ is odd, $f$ does the same to
$S^\vell(H^1(\Gamma)/V)$.  We can choose $f$-invariant orientations of the
vertices of $\Gamma$ that are not fixed by $f$.  Since $f$ then reverses $n$
vertex orientations, it negates $\tGamma$ as a Jacobi diagram, so we obtain
\[ \tGamma = -\tGamma = 0 \in \cP_{g,\vell}. \qedhere \]
\end{proof}

\begin{corollary} $\cP_{g,\vell} = 0$ when $\vell$ is odd
and $g \le 2$.
\label{c:odd12} \end{corollary}

\Cor{c:odd12} is elementary; it doesn't even need the IHX relator, only
AS and CL.  Other special cases of \Lem{l:symvan} with $V = 0$ are also
fairly elementary.  However, we will also use \Lem{l:symvan} when $V$ is
only part of $H^1(\Gamma)$, after moving the legs of $\tGamma$ into $V$
using other arguments.

The next lemmas are about Jacobi diagrams that have a bridge or a Y-bridge,
as in these examples:
\[ \begin{tp}[style=basic]
\draw (0,0) circle (1);
\foreach \th in {72,144,...,288} { \draw (\th:1) -- (\th:1.8); }
\draw (1,0) -- (2.5,0);
\tikzset{shift={(4.5,0)}};
\draw (0,0) ellipse (2 and 1.414) (0,-1.414) -- (0,1.414);
\draw (0,0) -- (.8,0);
\draw ({sqrt(8/3)},{-sqrt(2/3)}) -- ++({.8/sqrt(2)},{-.8/sqrt(2)});
\draw ({sqrt(8/3)},{sqrt(2/3)}) -- ++({.8/sqrt(2)},{.8/sqrt(2)});
\tikzset{shift={(9,0)}};
\foreach \th in {0,120,240} { \draw (0,0) -- (\th:2) (\th:3) circle (1); }
\foreach \th in {-120,-60,...,120} { \draw (0:3)++(\th:1) -- ++(\th:.8); }
\foreach \th in {30,120,210} { \draw (120:3)++(\th:1) -- ++(\th:.8); }
\draw (240:4) -- (240:4.8);
\end{tp} \]

\begin{lemma}[Bridge vanishing] Jacobi diagrams with bridges vanish in
the following two cases.  In both cases, let $g \ge 1$ be a loop number,
let $\vell$ be a leg color population, and let $\tGamma \in \cP_{g,\vell}$
be a Jacobi diagram.
\begin{enumerate}
\item If $\tGamma$ has a Y-bridge, then it vanishes for any $g$ and $\vell$.
\item Suppose that $\cP_{h,(1,\vk)} = 0$ whenever $h < g$, $\vk$ is
even, and $\vk \preceq \vell$ in the sense of inclusion of multisets.
Then $\tGamma$ vanishes if it has a bridge and $\vell$ is odd.
\end{enumerate}
\label{l:bridge} \end{lemma}

\begin{proof} In part 1, we can decompose $\tGamma$ as $\tGamma =
\tGamma_1\consum\tGamma_2$, where $\tGamma_1$ attaches to $\tGamma_2$
at a leg $y$ on a bridge of $\Gamma$.  Since $y$ is null-homologous,
$\tGamma_2$ vanishes in its Jacobi space, thus $\tGamma$ also vanishes.

In part 2, $\tGamma = \tGamma_1\consum\tGamma_2$ with an odd total number
of legs in exactly one of $\tGamma_1$ or $\tGamma_2$, without loss of
generality $\tGamma_1$.  Then by hypothesis, $\tGamma_1 \in \cP_{h,(1,\vk)}
= 0$ vanishes.

\end{proof}

The final two lemmas assert that under certain hypotheses that will be
provided by other vanishing lemmas, $\tGamma$ vanishes when it has a leg
next to a digon or a triangle (and the digon or triangle might also have
legs on them).

\begin{lemma}[Leg-by-digon vanishing] Suppose that for some fixed $\ell$
and $g$, a Jacobi diagram $\tGamma \in \cP_{g,\ell}$ vanishes whenever
it has a bridge.  Then $\tGamma \in \cP_{g,\ell}$ also vanishes if it has
a leg next to a digon (possibly with its own legs).
\label{l:lbd} \end{lemma}

\begin{proof} We can localize to the case when the skeleton $\Gamma$
is a digon with two anchors and only one leg next to the digon, if we
include bridge vanishing by fiat as an extra relator.  In other words,
we want to show that
\begin{eq}{e:digon} \begin{tp}[style=basic]
\draw (-3,0) -- (-1.25,0) (0,0) circle (1.25) (1.25,0) -- (3.75,0);
\draw (-2.8,0) node[left] {\anchvert} (3.55,0) node[right] {\anchvert};
\foreach \th in {45,117,135,-45,-117,-135} { \draw (\th:1.25) -- (\th:2.05); }
\foreach \th in {63,81,99,-63,-81,-99} { \fill (\th:1.65) circle (.08); }
\draw (2.5,0) -- (2.5,.8);
\draw (0,2.5) node {$a$} (0,-2.5) node {$b$};
\end{tp} = 0 \in \cP_{1,2,a+b+1}/\braket{\Br} \end{eq}
for all $a,b \ge 0$, where Br is the set of Jacobi diagrams with bridges.
Working with weight polynomials, we can flip the digon and then apply IHX
at the indicated edge:
\begin{multline*} \begin{tp}[style=basic,very thick]
\draw (-3.5,0) -- (-1.5,0) (0,0) circle (1.5) (1.5,0) -- (3.5,0);
\draw (-3.3,0) node[left] {\anchvert} (3.3,0) node[right] {\anchvert};
\draw[medcyan,->] (-170:1.1) arc (-170:170:1.1);
\draw[medcyan,->] (-3.5,.4) -- ({180-atan(.4/1.9)}:1.9)
    arc ({180-atan(.4/1.9)}:{atan(.4/1.9)}:1.9) -- (3.5,.4);
\draw (0,2.5) node {$x_1$} (0,0) node {$x_2$};
\end{tp} \stackrel{\text{flip}}=
\begin{tp}[style=basic,very thick]
\draw (-3.5,0) -- (-1.5,0) (0,0) circle (1.5) (1.5,0) -- (3.5,0);
\draw[medyel,ultra thick] (-1.5,0) arc (180:360:1.5);
\draw (-3.3,0) node[left] {\anchvert} (3.3,0) node[right] {\anchvert};
\draw[medcyan,->] (-170:1.1) arc (-170:170:1.1);
\draw[medcyan,->] (-3.5,.4) -- ({180-atan(.4/1.9)}:1.9)
    arc ({180-atan(.4/1.9)}:{atan(.4/1.9)}:1.9) -- (3.5,.4);
\draw (0,2.5) node {$x_1+x_2$} (0,0) node {$-x_2$};
\end{tp} \\
\stackrel{\IHX}= \begin{tp}[style=basic,very thick]
\draw (-3.5,0) -- (-1.5,0) (0,0) circle (1.5) (1.5,0) -- (3.5,0);
\draw (-3.3,0) node[left] {\anchvert} (3.3,0) node[right] {\anchvert};
\draw[medcyan,->] (-170:1.1) arc (-170:170:1.1);
\draw[medcyan,->] (-3.5,.4) -- ({180-atan(.4/1.9)}:1.9)
    arc ({180-atan(.4/1.9)}:{atan(.4/1.9)}:1.9) -- (3.5,.4);
\draw (0,2.5) node {$x_1+x_2$} (0,0) node {$x_2$};
\end{tp} + \begin{tp}[style=basic,very thick]
\draw (-3,2) -- (3,2) (0,2) -- (0,.5) (0,-1) circle (1.5);
\draw[medcyan,->] (-3,2.4) -- (3,2.4);
\draw[medcyan,->] (-170:1.1)++(0,-1) arc (-170:170:1.1);
\draw (0,3.1) node {$x_1+x_2$} (0,-1) node {$-x_2$};
\draw (-2.8,2) node[left] {\anchvert} (2.8,2) node[right] {\anchvert};
\draw[darkred,ultra thick] (-3.25,-2.5) -- (3,3.75);
\end{tp}
\end{multline*}

Thus the weight polynomial $w_\Gamma(x_1,x_2)$ satisfies
\[ w_\Gamma(x_1,x_2) = w_\Gamma(x_1+x_2,x_2), \]
which is only possible if
\[ w_\Gamma(x_1,x_2) = w_\Gamma(x_2) \]
is independent of $x_1$.  All of the terms in $w_\Gamma(x_1,x_2)$ that
use $x_1$ vanish, which is dually equivalent to the statement that every
Jacobi diagram with a leg by the digon vanishes.
\end{proof}

\begin{remark} As far as we know, it is simpler to prove \Lem{l:lbd}
(and \Lem{l:lbt} below) with weight polynomials than directly with
Jacobi diagrams with legs.  The author also found explicit sequences
of AS and IHX moves to show that all Jacobi diagrams \eqref{e:digon}
vanish if bridges vanish.  Our direct proof is by induction on $a+b$.
We invite the reader to explore the matter independently.
\end{remark}

\begin{lemma}[Leg-by-triangle vanishing] Suppose that for some fixed $\ell$
and $g$, a Jacobi diagram $\tGamma \in \cP_{g,\ell}$ vanishes whenever
its skeleton $\Gamma$ has a digon.  Then $\tGamma \in \cP_{g,\ell}$ also
vanishes if it has a leg next to a triangle (possibly with its own legs).
\label{l:lbt} \end{lemma}

\begin{proof} As in \Lem{l:lbd}, we can localize to the case when $\Gamma$
is a triangle with three anchors, if we include digon vanishing by fiat.
Letting $\Di$ be the set of diagrams with a digon, we want to show that
\[ \tGamma = 0 \in \cP_{1,3,\ell}/\braket{\Di} \]
when $\tGamma$ has a leg next to the triangle (and possibly legs on the
triangle).  Again working with weight polynomials, we can apply the IHX
relation twice at the indicated edges, omitting digon terms that vanish
by hypothesis:
\begin{align*} \begin{tp}[style=basic,very thick,shift={(0,-.8)}]
\draw (90:2.8) -- (210:2.8) -- (330:2.8) -- cycle;
\draw[medyel,ultra thick] (210:2.8) -- (330:2.8);
\foreach \th in {90,210,330} {
    \draw (\th:2.8) -- (\th:4); \draw (\th:4.6) node {\anchvert};}
\draw[medcyan,rounded corners,->,rotate=240]
    (-.15,-1) -- (210:2) -- (90:2) -- (330:2) -- (.15,-1);
\foreach \th in {90,330}
    { \draw[medcyan,->,rounded corners]
    ($(\th+120:4)+(\th+30:.4)$) -- ($(\th+120:2.8)+(\th+45:{.4/cos(15)})$)
    -- ($(\th:2.8)+(\th+75:{.4/cos(15)})$) -- ($(\th:4)+(\th+90:.4)$); }
\draw (0,0) node {$x_3$};
\draw (150:2.6) node {$x_1$} (30:2.6) node {$x_2$};
\end{tp} &\stackrel{\IHX}= -
\begin{tp}[style=basic,very thick,shift={(0,-.8)}]
\draw (90:2.8) -- (210:2.8) -- (330:2.8) -- cycle;
\draw[medyel,ultra thick] (210:2.8) -- (90:2.8);
\foreach \th in {90,210,330} {
    \draw (\th:2.8) -- (\th:4); \draw (\th:4.6) node {\anchvert};}
\draw[medcyan,rounded corners,->,rotate=240]
    (-.15,-1) -- (210:2) -- (90:2) -- (330:2) -- (.15,-1);
\foreach \th in {90,330}
    { \draw[medcyan,->,rounded corners]
    ($(\th+120:4)+(\th+30:.4)$) -- ($(\th+120:2.8)+(\th+45:{.4/cos(15)})$)
    -- ($(\th:2.8)+(\th+75:{.4/cos(15)})$) -- ($(\th:4)+(\th+90:.4)$); }
\draw (.15,0) node {$-x_3$};
\draw (150:2.6)++(-.4,1) node {$x_1+x_3$} (30:2.6)++(.4,1) node {$x_2+x_3$};
\end{tp} \\ &\stackrel{\IHX}=
\begin{tp}[style=basic,very thick,shift={(0,-.8)}]
\draw (90:2.8) -- (210:2.8) -- (330:2.8) -- cycle;
\foreach \th in {90,210,330} {
    \draw (\th:2.8) -- (\th:4); \draw (\th:4.6) node {\anchvert};}
\draw[medcyan,rounded corners,->,rotate=240]
    (-.15,-1) -- (210:2) -- (90:2) -- (330:2) -- (.15,-1);
\foreach \th in {90,330}
    { \draw[medcyan,->,rounded corners]
    ($(\th+120:4)+(\th+30:.4)$) -- ($(\th+120:2.8)+(\th+45:{.4/cos(15)})$)
    -- ($(\th:2.8)+(\th+75:{.4/cos(15)})$) -- ($(\th:4)+(\th+90:.4)$); }
\draw (0,0) node {$x_3$};
\draw (150:2.6)++(-.4,1) node {$x_1+x_3$} (30:2.6) node {$x_2$};
\end{tp} \end{align*}
Thus the weight polynomial $w_\Gamma(x_1,x_2,x_3)$ satisfies
\[ w_\Gamma(x_1,x_2,x_3) = w_\Gamma(x_1+x_3,x_2,x_3) = w_\Gamma(x_2,x_3); \]
it cannot depend on $x_1$.  The same calculation with left and
right swapped in each triangle yields
\[ w_\Gamma(x_1,x_2,x_3) = w_\Gamma(x_1,x_2+x_3,x_3) = w_\Gamma(x_1,x_3). \]
Thus $w_\Gamma$ can only depend on $x_3$, and the equivalent dual statement
for Jacobi diagrams is the statement of the lemma. \end{proof}

\section{Three loops}
\label{s:three}

In this section, we reprove the Moskovich--Ohtsuki result that odd Jacobi
diagrams with 3 loops vanish.

Let $\ell$ be odd, and let $\tGamma \in \cP_{3,\ell}$ be a Jacobi diagram
with 3 loops and $\ell$ legs.  By \Cor{c:odd12} (odd vanishing for $\le 2$
loops) and \Lem{l:bridge} (bridge vanishing), we can assume that $\tGamma$
is bridgeless.  Thus there are two choices for its skeleton $\Gamma$
(before orienting the vertices):
\[ \Gamma_1 =\; \begin{tp}[style=basic]
\draw (0,0) ellipse (2.8 and 1.4);
\foreach \th in {0,180} { \draw[rotate=\th] ({2.8*cos(65)},{1.4*sin(65)})
    to[bend right=20] ({2.8*cos(65)},{-1.4*sin(65)}); } \end{tp}
\qquad \Gamma_2 =\; \begin{tp}[style=basic]
\draw (0,0) circle (2);
\foreach \th in {90,210,330} { \draw (0,0) -- (\th:2); }
\end{tp} \]
If $\Gamma = \Gamma_1$, then $\tGamma = 0$ by \Lem{l:symvan} (symmetry
vanishing). If $\Gamma = \Gamma_2$, then $\tGamma = 0$ by \Lem{l:lbt} (leg
by triangle), because each of the 6 edges of the tetrahedron $\Gamma_2$
is next to a triangle, and because we already showed that diagrams with
digons vanish.  Thus $\cP_{3,\ell} = 0$.

We will use the following generalization at 3 loops to establish
bridge vanishing at 4 loops.

\begin{theorem} If $\vell = (\ell_1,\ell_2)$ is odd, then $\cP_{3,\vell}
= 0$.
\label{th:odd3c2} \end{theorem}

The above proof that $\cP_{3,\ell} = 0$ when $\ell$ is odd does not establish
\Thm{th:odd3c2}, because Lemmas~\ref{l:lbd} and \ref{l:lbt} are only valid
as stated for polynomials in $H_1(\Gamma)$ and not more general tensors.
Thus the case $\Gamma = \Gamma_2$ requires a different argument.

\begin{lemma} Ignoring the other skeleta with 3 loops, the six IHX operations
on the tetrahedron $\Gamma = \Gamma_2$ generate the action of $\SL(3,\Z)$ on
\[ H^1(\Gamma) \cong V \defeq \Q^3. \]
Moreover, for every odd color population $\vell$, the IHX operations
generate the induced action of $\SL(3,\Z)$ on
\[ \cL_\vell(\Gamma) \cong S^\vell(V). \]
\label{l:sl3} \end{lemma} \eatline \eatline

\begin{proof} In the diagram \eqref{e:tetra}, recall that the homology
basis $\{x_k\}$ can be viewed as linear functions on $H^1(\Gamma)$,
indeed as the scalar coefficients of a vector $v \in H^1(\Gamma)$ in the
cohomology basis $\{e_k\}$.  If we apply the two-step calculation in the
proof of \Lem{l:lbt} to the tetrahedron diagram \eqref{e:tetra}, the result
is 2 of the 6 elementary matrices that express adding one row to another:
\begin{align*} \begin{bmatrix} x'_1 \\ x'_2 \\ x'_3 \end{bmatrix}
    &= \begin{bmatrix} 1 & 0 & 1 \\ 0 & 1 & 0 \\ 0 & 0 & 1 \end{bmatrix}
    \begin{bmatrix} x_1 \\ x_2 \\ x_3 \end{bmatrix} \in V \\[1ex]
    \begin{bmatrix} x'_1 \\ x'_2 \\ x'_3 \end{bmatrix}
    &= \begin{bmatrix} 1 & 0 & 0 \\ 0 & 1 & 1 \\ 0 & 0 & 1 \end{bmatrix}
    \begin{bmatrix} x_1 \\ x_2 \\ x_3 \end{bmatrix} \in V.
\end{align*}
A rotation of the tetrahedron cyclically permutes the basis $\{e_k\}$.
If we conjugate these two elementary matrices by such a rotation, we get
all 6 elementary matrices, which then generate $\SL(3,\Z)$.

Thus the IHX operations generate at least $\SL(3,\Z)$.  We also want to
show that any one IHX operation lies in $\SL(3,\Z)$.  Any such operation,
for example either IHX step in the proof of \Lem{l:lbt}, yields a change
of basis $f \in \GL(3,\Z)$ of $H^1(\Gamma)$ with determinant $-1$, and
also negates $\Gamma$.  When $\vell$ is odd, the result is equivalent
to acting by $-f \in \SL(3,\Z)$, so that every IHX operation acts on
$S^\vell(H^1(\Gamma))$ by an element of $\SL(3,\Z)$.
\end{proof}

\begin{proof}[Proof of \Thm{th:odd3c2}] By \Cor{c:odd12} (odd vanishing
for 1 and 2 loops) and \Lem{l:bridge} (bridge vanishing), we can again
assume that $\tGamma \in \cP_{3,\ell}$ is bridgeless.  If $\Gamma =
\Gamma_1$, then $\tGamma = 0$ again by \Lem{l:symvan} (symmetry vanishing).
Only the tetrahedron case $\Gamma = \Gamma_2$ requires a different argument,
because \Lem{l:lbt} (leg by triangle) does not hold for colored legs.

If $\Gamma = \Gamma_2$, then \Lem{l:sl3} implies that
$\cP_{3,(\ell_1,\ell_2)}$ is the coinvariant quotient
\[ \cP_{3,(\ell_1,\ell_2)} \cong \Coinv_{\SL(3,\Z)}
    (S^{\ell_1}(V) \otimes S^{\ell_2}(V)). \]
Note that
\[ W \defeq S^{\ell_1}(V) \otimes S^{\ell_2}(V) \]
is a representation of $\SL(3,\Q)$ with polynomial matrix entries, that
$\SL(3,\Z)$ is Zariski dense in $\SL(3,\Q)$, and that $\SL(3,\Q)$ is a
reductive algebraic group.  It follows that
\[ \Coinv_{\SL(3,\Z)}(W) = \Coinv_{\SL(3,\Q)}(W) \cong \Inv_{\SL(3,\Q)}(W). \]
The theorem holds if and only if the direct sum decomposition of $\SL(3,\Q)
\acts W$ does not include the trivial representation.  It does not, because
$S^\ell(V)$ is the irrep of $\SL(3,\Q)$ of highest weight $(\ell,0)$,
and two such irreps are only dual when $\ell_1$ and $\ell_2$ both vanish,
which is not the case here.  Thus $\cP_{3,(\ell_1,\ell_2)}$ vanishes.
\end{proof}

\begin{remark} The proof of \Thm{th:odd3c2} can be extended to
show that $\cP_{3,\vell}$ vanishes when $\vell =
(\ell_1,\ell_2,\ell_3)$ is odd and $\ell_1$,$\ell_2$, and $\ell_3$ are
not all equal. It also shows that $\dim(\cP_{3,(\ell,\ell,\ell)}) = 1$
when $\ell$ is odd; in particular, it does not vanish.
\end{remark}

\section{Four loops}

In this section, we prove \Thm{th:odd4}.

Let $\ell$ be odd and let $\tGamma \in \cP_{4,\ell}$ be a Jacobi diagram
with 4 loops and $\ell$ legs. \Cor{c:odd12} establishes vanishing for 1
or 2 loops for any odd color population $\vell$ of legs.  \Thm{th:odd3c2}
establishes odd vanishing for 3 loops, for any odd color population $\vell =
(\ell_1,\ell_2)$ with two colors.  Thus \Lem{l:bridge} holds, and tells us
that every $\tGamma \in \cP_{4,\ell}$ with a bridge vanishes.  This leaves
5 bridgeless choices for the skeleton $\Gamma$ (before orienting vertices).
We label the first 4 with homology generators, omitting a given generator
$x_k$ when we use \Lem{l:lbd} or \ref{l:lbt} to remove the dependence of
the weight polynomial $w_\Gamma(\vx)$ on $x_k$.

The cases (which we will settle in order) are as follows.
\begin{enumerate}
\item If
\[ \Gamma = \Gamma_1 =\;\; \begin{tp}[style=basic]
\draw (0,0) circle (2);
\foreach \th/\n in {0/2,120/1,240/3} {
    \draw[fill=white] (\th:2) circle (.9);
    \draw (\th:2) node {$x_\n$}; } \end{tp}\;\;, \]
then $\tGamma$ vanishes by \Lem{l:lbd} (leg by digon) followed by
\Lem{l:symvan} (symmetry vanishing).
\item If
\[ \Gamma = \Gamma_2 =\;\; \begin{tp}[style=basic]
\draw (0,0) ellipse (4 and 1.5) (0,-1.5) -- (0,1.5);
\draw ({4*cos(55)},{1.5*sin(55)}) to[bend right=25]
    ({4*cos(55)},{-1.5*sin(55)});
\draw ({-4*cos(55)},{1.5*sin(55)}) to[bend left=25]
    ({-4*cos(55)},{-1.5*sin(55)});
\foreach \x/\n in {-3/1,-1/2,1/3,3/4} { \draw (\x,0) node {$x_\n$}; }
\end{tp}\;\;, \]
then $\tGamma$ vanishes by \Lem{l:symvan}
(symmetry vanishing).
\item If
\[ \Gamma = \Gamma_3 =\;\; \begin{tp}[style=basic]
\draw (0,0) circle (2.5);
\foreach \th in {80,220,330} { \draw (0,0) -- (\th:2.5); }
\draw[fill=white] (150:2.5) circle (.9);
\draw (150:2.5) node {$x_1$} (25:1.5) node {$x_2$} (275:1.5) node {$x_3$};
\end{tp}\;\;, \]
then $\tGamma$ vanishes by \Lem{l:lbd} (leg by digon)
followed by \Lem{l:symvan} (symmetry vanishing).
\item  If
\[ \Gamma = \Gamma_4 =\;\; \begin{tp}[style=basic]
\draw (0,0) ellipse (4 and 2);
\foreach \s in {-1,1} { \draw ({\s*4*cos(60)},{2*sin(60)}) -- ({\s*1.3},0)
    -- ({\s*4*cos(60)},{-2*sin(60)}); }
\draw (-1.3,0) -- (1.3,0);
\draw (-2.7,0) node {$x_1$} (2.7,0) node {$x_2$};
\end{tp}\;\;, \]
then $\tGamma$ vanishes by \Lem{l:lbt} (leg by triangle) followed by
\Lem{l:symvan} (symmetry vanishing), given that the earlier cases show
that digons vanish.
\item If
\[ \Gamma = \Gamma_5 = K_{3,3} =\;\;\begin{tp}[style=basic]
\foreach \th in {0,120,240} {
    \draw ({\th+348}:2.5) -- ({\th+270}:{2.5*sin(12)}); }
\foreach \th in {0,120,240} {
    \draw[style=overcross] ({\th+270}:{2.5*sin(12)}) -- ({\th+192}:2.5); }
\draw (0,0) circle (2.5);
\end{tp}\;\;, \]
then any IHX move yields two terms with skeleton $\Gamma_4$, thus $\tGamma$
vanishes.
\end{enumerate}
This concludes the proof of \Thm{th:odd4}.

\section{Five or more loops}

In this article, we proved \Thm{th:odd4} without complicated formulas or
computer assistance.  It is not clear whether odd vanishing for 5 loops
is practical without computer assistance. For 6 or more loops, computer
assistance is presumably necessary or at least highly advisable, unless
a deeper idea (or a counterexample!) arises to better understand the odd
Vassiliev problem.

Let $\ell$ be odd.  Using our methods, the odd Vassiliev problem for 5 loops
begins with the question of whether every diagram $\tGamma \in \cP_{5,\ell}$
with a bridge vanishes.  We do not know whether $\cP_{4,(1,\ell-1)}$
always vanishes.  If it does, then bridges vanish in $\cP_{5,\ell}$ by
\Lem{l:bridge}.  Since $\cP_{3,(1,1,1)}$ does not vanish, we do not even
know whether bridges vanish in $\cP_{4,(1,\ell-1)}$; although only one
skeleton remains as an open case for that question, namely a tetrahedron
with a tadpole.  If bridges vanish in $\cP_{5,\ell}$, then that leaves
16 bridgeless choices for the skeleton $\Gamma$.  We know that $\tGamma$
vanishes in some of these cases, but enough cases remain that chaos could
ensue as far as we know.

For any fixed $g$ and $\ell$, we can always decompose a weight system $w \in
\cP_{g,\ell}^*$ into a finite list of weight polynomials $w_\Gamma(\vx)$
of degree $\ell$ (including when $\Gamma$ has a bridge if necessary).
As stated in \Sec{ss:homology}, each $w_\Gamma(\vx)$ must transform to
itself under each symmetry of $\Gamma$, sometimes with a sign correction.
More explicitly, each symmetry of $\Gamma$ yields an equation of the form
\[ w_\Gamma(\vx) = \pm w_\Gamma(f(\vx)) \]
with $f \in \GL(g,\Z)$.  Likewise each IHX relation yields an equation
\[ w_{\Gamma_1}(f_1(\vx)) \pm w_{\Gamma_2}(f_2(\vx))
    \pm w_{\Gamma_3}(f_3(\vx)) = 0 \]
with each $f_k \in \GL(g,\Z)$.  Dually and more abstractly, we obtain
a characterization of the form
\[ \cP_{g,\ell} \cong M \tensor_A S, \]
where:
\begin{enumerate}
\item $A = \Q[\GL(g,\Z)]$ is the group algebra of $\GL(g,\Z)$.
\item $S = S^\ell(\Q^g)$ is a left $A$-module because it is a linear
representation of $\GL(g,\Z)$.
\item $M$ is a finitely presented right $A$-module that depends only on
the loop number $g$.  It is defined by automorphisms of trivalent graphs
and IHX and AS relators.
\end{enumerate}
In this article, we whittled this system of equations down to nothing when
$g \le 4$ and $\ell$ is odd with the aid of vanishing lemmas.  We believe
that a computer algebra calculation could provide similar progress for
the next several values of $g$.

Finally, partial progress towards odd vanishing could also help uncover
a counterexample.  For example, as remarked at the end of \Sec{s:three},
\[ \dim(\cP_{3,(\ell,\ell,\ell)}) = 1 \]
when $\ell$ is odd.  With the methods in this article, it is easier to
see what does not vanish in $\cP_{3,(\ell,\ell,\ell)}$ by first
learning what does vanish.

\bibliography{gr,gt,qa,me}

\end{document}